\documentclass[11pt]{amsart}
\usepackage{lmodern}
\usepackage{amsmath, amsthm, amssymb, amsfonts}
\usepackage[normalem]{ulem}
\usepackage{hyperref}

\usepackage{mathrsfs}

\usepackage{verbatim} 
\usepackage{longtable}

\usepackage{mathtools}

\usepackage{tikz}
\usetikzlibrary{decorations.pathmorphing}
\tikzset{snake it/.style={decorate, decoration=snake}}

\usepackage{caption}

\usepackage{tikz-cd}
\usetikzlibrary{arrows}

\theoremstyle{plain}
\newtheorem{thm}{Theorem}[section]
\newtheorem{cor}[thm]{Corollary}
\newtheorem{lem}[thm]{Lemma}
\newtheorem{prop}[thm]{Proposition}

\theoremstyle{definition}

\theoremstyle{remark}
\newtheorem{rmk}[thm]{Remark}

\newcommand{\BC}{{\mathbb{C}}}

\newcommand{\BG}{{\mathbb{G}}}

\newcommand{\BN}{{\mathbb{N}}}

\newcommand{\BP}{{\mathbb{P}}}

\newcommand{\BZ}{{\mathbb{Z}}}

\newcommand{\CC}{{\mathcal C}}

\newcommand{\CE}{{\mathcal E}}
\newcommand{\CF}{{\mathcal F}}

\newcommand{\CH}{{\mathcal H}}

\newcommand{\CK}{{\mathcal K}}
\newcommand{\CL}{{\mathcal L}}
\newcommand{\CM}{{\mathcal M}}

\newcommand{\CO}{{\mathcal O}}
\newcommand{\CP}{{\mathcal P}}

\newcommand{\CS}{{\mathcal S}}

\newcommand{\CW}{{\mathcal W}}
\newcommand{\CX}{{\mathcal X}}

\newcommand{\CZ}{{\mathcal Z}}

\newcommand{\BBC}{{\underline{\BC}}}

\DeclareFontFamily{OT1}{rsfs}{}
\DeclareFontShape{OT1}{rsfs}{n}{it}{<-> rsfs10}{}
\DeclareMathAlphabet{\curly}{OT1}{rsfs}{n}{it}

\newcommand{\Coh}{\mathrm{Coh}}

\usepackage{tikz}
\usepackage{lmodern}
\usetikzlibrary{decorations.pathmorphing}

\begin{document}
\title[On the intersection cohomology of the moduli of $\mathrm{SL}_n$- Higgs bundles]{On the intersection cohomology of the moduli of $\mathrm{SL}_n$-Higgs bundles on a curve}
\date{\today}

\author[D. Maulik]{Davesh Maulik}
\address{Massachusetts Institute of Technology}
\email{maulik@mit.edu}

\author[J. Shen]{Junliang Shen}
\address{Massachusetts Institute of Technology}
\email{jlshen@mit.edu}
\address{Yale University}
\email{junliang.shen@yale.edu}

\begin{abstract}
We explore the cohomological structure for the (possibly singular) moduli of $\mathrm{SL}_n$-Higgs bundles for arbitrary degree on a genus $g$ curve with respect to an effective divisor of degree $>2g-2$. We prove a support theorem for the $\mathrm{SL}_n$-Hitchin fibration extending de Cataldo's support theorem in the nonsingular case, and a version of the Hausel--Thaddeus topological mirror symmetry conjecture for intersection cohomology. This implies a generalization of the Harder--Narasimhan theorem concerning semistable vector bundles for any degree.

Our main tool is an Ng\^{o}--type support inequality established recently which works for possibly singular ambient spaces and intersection cohomology complexes.

\end{abstract}

\maketitle

\setcounter{tocdepth}{1} 

\tableofcontents
\setcounter{section}{-1}

\section{Introduction}

\subsection{Overview}
Throughout, we work over the complex numbers $\BC$. Let $C$ be a nonsingular irreducible projective curve of genus $g \geq 2$. The purpose of this paper is to explore cohomological structures for the moduli space of degree $d$ semistable $\mathrm{SL}_n$-Higgs bundles on $C$ with respect to an effective divisor $D$ of degree $\mathrm{deg}(D)>2g-2$. More precisely, we show that the support theorem \cite{dC_SL} and the topological mirror symmetry conjecture \cite{HT, GWZ, MS}, which were proven in the case $\mathrm{gcd}(n,d)=1$, actually hold for \emph{arbitrary} $d$.

For this more general setting, the essential difference with the coprime case is that the moduli space may be singular due to the presence of strictly semistable locus. Hence it is natural for us to consider intersection cohomology. Our main tool is an Ng\^o--type support inequality for weak abelian fibrations recently established in \cite{MS2} which works for singular ambient spaces and intersection cohomology complexes. 

As an immediate application of our results, we also give a proof of a generalized version of the Harder--Narasimhan theorem \cite{HN} for intersection cohomology and arbitrary degree.

\subsection{Moduli of $\mathrm{SL}_n$-Higgs bundles}
We fix $D$ to be an effective divisor of degree $\mathrm{deg}(D)>2g-2$ and we fix $L \in \mathrm{Pic}^d(C)$ to be a degree $d$ line bundle on $C$. We denote by $M_{n,L}$ the moduli space of semistable Higgs bundles 
\[
(\CE, \theta): \quad \theta: \CE \to \CE\otimes \CO_C(D), \quad \mathrm{rank}(\CE)=n, \quad \mathrm{det}(\CE)\simeq L, \quad \mathrm{trace}(\theta) = 0,
\]
where the (semi-)stability is with respect to the slope $\mu(\CE,\theta) = \mathrm{deg}(\CE)/\mathrm{rank}(\CE)$. The moduli space $M_{n,L}$ admits a proper surjective morphism 
\begin{equation}\label{Hitchin}
h: M_{n,L} \to A = \bigoplus_{i= 2}^n H^0(C, \CO_C(iD)), \quad (\CE ,\theta) \mapsto \mathrm{char}(\theta)
\end{equation}
known as the Hitchin fibration \cite{Hit, Hit1}. Here $\mathrm{char}(\theta)$ denotes the characteristic polynomial of the Higgs field $\theta: \CE \to \CE \otimes \CO_C(D)$:
\[
\mathrm{char}(\theta) = (a_2, a_3, \dots, a_n), \quad a_i = \mathrm{trace}(\wedge^i\theta) \in H^0(C, \CO_C(iD)).
\]
Alternatively, we may view a closed point $a\in A$ as a spectral curve $C_a \subset \mathrm{Tot}_C(\CO_C(D))$ which is a degree $n$ cover over the zero section $C$. Let the elliptic locus $A^{\mathrm{ell}} \subset A$ be the open subset consisting of \emph{integral} spectral curves. The fibers of the restricted Hitchin fibration over $A^{\mathrm{ell}}$
\begin{equation}\label{eqn8}
h^{\mathrm{ell}}: M_{n,L}^{\mathrm{ell}} \rightarrow A^{\mathrm{ell}}
\end{equation}
are compactified Prym varieties of the integral spectral curves $C_a$. In particular, the open subvariety $M_{n,L}^{\mathrm{ell}}$ is nonsingular and contained in the stable locus $M_{n,L}^s$:
\[
M_{n,L}^{\mathrm{ell}} \subset M_{n,L}^s \subset M_{n,L}.
\]

\subsection{Support theorem for $\mathrm{SL}_n$}
By \cite{BBD}, we have the decomposition for the direct image complex of the intersection cohomology complex
\[
Rh_* \mathrm{IC}_{M_{n,L}} \simeq \bigoplus_{\alpha,i} \mathrm{IC}_{Z_{\alpha,i}}(\CL_{\alpha,i})[-r_i] \in D^b_c(A), \quad r_i \in \BZ
\]
into (shifted) simple perverse sheaves. Here $D^b_c(-)$ denotes the bounded derived category of constructible sheaves, $Z_{\alpha,i} \subset A$ are irreducible closed subvarieties, each $\CL_{\alpha,i}$ is a simple local system on an open subset of $Z_{\alpha.i}$, and $\mathrm{IC}_{Z_{\alpha,i}}(\CL_{\alpha,i})$ is the intermediate extension of $\CL_{\alpha,i}$ in $Z_{\alpha,i}$. We call $Z_{\alpha,i}$ the \emph{supports} of the direct image complex $Rh_*\mathrm{IC}_{M_{n,L}}$ that are important invariants for the map $h: M_{n,L} \to A$.

The following theorem, which generalizes de Cataldo's $\mathrm{SL}_n$-support theorem \cite{dC_SL} in the case of $\mathrm{gcd}(n,d)=1$, shows that the decomposition theorem of the Hitchin fibration $h: M_{n,L} \to A$ is governed by the elliptic locus (\ref{eqn8}).

\begin{thm}[Support theorem]\label{thm0.1} Assume that $M_{n,L}$ is the moduli space associated with an effective divisor $D$ with $\mathrm{deg}(D)>2g-2$. The generic point of any support of ${Rh}_* \mathrm{IC}_{M_{n,L}}$ lies in the elliptic locus $A^{\mathrm{ell}}$.
\end{thm}

In fact, by combining the techniques of \cite{CL, dC_SL} and \cite{MS2}, we prove in Sections \ref{sec1} and \ref{Sec2} a more general support theorem (Theorem \ref{thm1.1}) for certain relative moduli space of Higgs bundles associated with a cyclic \'etale Galois cover $\pi: C' \to C$. These moduli spaces are tightly connected to the endoscopic theory for $\mathrm{SL}_n$ \cite{Ngo0,Ngo} and the topological mirror symmetry for Hitchin systems \cite{HT, GWZ, MS}.

\subsection{Topological mirror symmetry}
Motivated by the Strominger--Yau--Zaslow mirror symmetry, Hausel--Thaddeus \cite{HT} conjectured that the moduli of semistable $\mathrm{SL}_n$- and $\mathrm{PGL}_n$-Higgs bundles should have identical (properly interpreted) Hodge numbers. In the case of $\mathrm{gcd}(n,d)=1$, the match of the Hodge numbers for the $\mathrm{SL}_n$- and $\mathrm{PGL}_n$-Higgs moduli spaces was formulated precisely in \cite{HT} using singular cohomology, and was proven recently in \cite{GWZ, LW, MS} by different methods. From the viewpoint of $S$-duality \cite[Section 5.4]{Survey} and the approach of \cite{MS}, the Hausel-Thaddeus conjecture is closely connected to the endoscopy theory and the fundamental lemma for $\mathrm{SL}_n$.

In this paper, we explore the Hausel--Thaddeus conjecture for arbitrary degree $d$. Under the assumption that $\mathrm{deg}(D)$ is even and greater than $2g-2$, we prove that an analog of the Hausel--Thaddeus conjecture holds for intersection cohomology and arbitrary degree $d$. Our approach follows the spirit of \cite{MS}, that we view the (refined) Hausel--Thaddeus conjecture \cite[Conjeture 4.5]{Survey} as an extension of Ng\^o's geometric stabilization theorem \cite{Ngo} in his proof of the fundamental lemma of the Langlands program. Our new input is the support theorem for $\mathrm{SL}_n$ and its endoscopic groups (see Theorem \ref{thm1.1}), relying on the framework of \cite{MS2}. 

Now in the following we introduce some notation and state the main theorem.

Let $\Gamma = \mathrm{Pic}^0(C)[n]$ be the group of $n$-torsion line bundles on $C$. The finite group $\Gamma$ admits a non-degenerate Weil pairing \cite[Section 1.3]{MS}, which after identifying $\Gamma$ with $H_1(C, \BZ/n\BZ)$, coincides with the intersection pairing. Hence we obtain a canonical isomorphism between $\Gamma$ and the group of characters $\Gamma = \mathrm{Hom}(\Gamma, \BG_m)$:
\begin{equation}\label{Weil_Pairing}
    \Gamma = \hat{\Gamma}.
\end{equation}

For the $\mathrm{SL}_n$-Higgs moduli space $M_{n,L}$ associated with the line bundle $L$, the corresponding $\mathrm{PGL}_n$-Higgs moduli space $[M_{n,L}/\Gamma]$ is a Deligne--Mumford stack obtained as the quotient of the natural finite group action of $\Gamma = \mathrm{Pic}^0(C)[n]$ on $M_{n,L}$:
\[
\CL \cdot (\CE, \theta) = (\CE\otimes \CL, \theta), \quad \quad \CL \in \Gamma,\quad (\CE, \theta) \in M_{n,L}.
\]
Note that when $\mathrm{gcd}(n,d) \neq 1$, both the $\mathrm{SL}_n$- and the $\mathrm{PGL}_n$-Higgs moduli spaces are singular as a variety and a Deligne--Mumford stack respectively. For an element $\gamma \in \Gamma$, we denote by $M^\gamma_{n,L} \subset M_{n,L}$ the subvariety of the $\gamma$-fixed locus. Assume that
\[
h_\gamma: M^\gamma_{n,L} \to A_\gamma : = \mathrm{Im}(h_\gamma) \subset A
\]
is the morphism induced by the Hitchin fibration (\ref{Hitchin}), which recovers $h$ when $\gamma = 0$. We denote by $i_\gamma: A_\gamma \hookrightarrow A$ the closed embedding and $d_\gamma$ the codimension of $A_\gamma$ in $A$. The $\Gamma$-action on $M_{n,L}$ induces a $\Gamma$-action on the fixed locus $M^\gamma_{n,L}$.  This action is fiberwise with respect to the morphism $h_\gamma$, which induces a canonical decomposition 
\[
{Rh_\gamma}_* \mathrm{IC}_{M^\gamma_{n,L}}  = \bigoplus_{\kappa} \left({Rh_\gamma}_* \mathrm{IC}_{M^\gamma_{n,L}}\right)_\kappa \in D^b_c(A_\gamma), \quad \kappa \in \hat{\Gamma}
\]
into eigen-subcomplexes \cite[Lemma 3.2.5]{NL}. The following theorem is a sheaf-theoretic version of the Hausel-Thaddeus conjecture for the divisor $D$, which resembles the fundamental lemma.

\begin{thm}\label{thm0.2}
Assume that $M_{n,L}$ is the moduli space associated with an effective divisor $D$ with $\mathrm{deg}(D)$ even and greater than $2g-2$. Assume that $\gamma \in \Gamma$ and $\kappa \in \hat{\Gamma}$ are matched via the Weil pairing (\ref{Weil_Pairing}). 
\begin{enumerate}
    \item[(a)](Endoscopic decomposition) We have an isomorphism 
    \begin{equation}\label{thm0.2_a}
     \left({Rh}_* \mathrm{IC}_{M_{n,L}}\right)_\kappa  \simeq {i_\gamma}_*\left({Rh_\gamma}_* \mathrm{IC}_{M^\gamma_{n,L}}\right)_\kappa [-2d_\gamma] \in D^b_c(A).
    \end{equation}
    \item[(b)](Transfer) Assume $L'\in \mathrm{Pic}^{d'}(C)$ with $\mathrm{gcd}(d,n)=\mathrm{gcd}(d',n)$. Then we have
    \[
    \left({Rh_\gamma}_*\mathrm{IC}_{M^\gamma_{n,L}}\right)_\kappa \simeq  \left({Rh_\gamma}_* \mathrm{IC}_{M^\gamma_{n,L'}}\right)_{q\kappa}  \in  D^b_c(A_\gamma)
    \]
    where $q$ is an integer coprime to $n$ satisfying that
    \begin{equation}\label{condition}
    d = d'q  \quad\mathrm{mod}~~~n. 
    \end{equation}
   
\end{enumerate}
Moreover, both (a) and (b) hold in the bounded derived categories $D^b\mathrm{MHM}(-)$ of mixed Hodge modules refining $D_c^b(-)$.
\end{thm}

Theorem \ref{thm0.2} concerns Higgs bundles with respect to $D$ satisfying that $\mathrm{deg}(D)$ is even and greater than $2g-2$.\footnote{When $\mathrm{deg}(D)$ is odd and greater than $2g-2$, Theorem \ref{thm0.2} also holds for odd rank $n$ by an identical proof; in fact applying the endoscopic correspondence in the proof of \cite[Theorem 3.3]{MS} only requires that $(n-1)\mathrm{deg}(D)$ is even.} By taking global cohomology, it recovers an identity between the (stringy) \emph{intersection} E-polynomials for the $\mathrm{SL}_n$- and the $\mathrm{PGL}_n$-Higgs moduli spaces:
\[
\mathrm{IE}(M_{n,L}; u,v) = \mathrm{IE}_{\mathrm{st},\mathrm{twisted}}\left([M_{n,L'}/\Gamma]; u,v \right).
\]
Here the intersection E-polynomial $\mathrm{IE}(-; u,v)$ is given in \cite[Section 1]{Mauri}, and the twisted stringy intersection E-polynomial is
\[
\mathrm{IE}_{\mathrm{st},\mathrm{twisted}}\left([M_{n,L'}/\Gamma]; u,v \right) = \sum_{\gamma \in \Gamma} \mathrm{IE}(M^\gamma_{n,L'};u,v)_{q\kappa} (uv)^{2d_\gamma};
\]
for each term on the righthand side the character $\kappa$ is matched with $\gamma$ via the Weil pairing and $q,L,L'$ are as in Theorem \ref{thm0.2}. This is analogous to the original Hausel--Thaddeus conjecture \cite{HT, Survey}. A natural question is if the intersection E-polynomial version of the Hausel--Thaddeus conjecture holds for $D= K_C$. This was recently conjectured by Mauri \cite{Mauri}, who also verified it for the case of $n=2$. We refer to Section \ref{HT_conj} for more discussions. 
 
\begin{rmk}
In \cite[Remark 3.30]{Survey}, Hausel proposed that a version of the topological mirror symmetry conjecture \cite{HT} should hold without the coprime assumption between the degrees and the rank, and he asked what is the cohomology theory we should use to formulate this. As mentioned above, Mauri proposed to use \emph{intersection cohomology}. Theorem \ref{thm0.2} provides further evidence that intersection cohomology is the correct theory to formulate the topological mirror symmetry for possibly singular moduli spaces. Our reasons come naturally from the decomposition theorem \cite{BBD} and the support theorem (Theorem \ref{thm1.1}).
\end{rmk}

\subsection{The Harder--Narasimhan theorem}\label{0.1}  The moduli space $N_{n,L}$ of (slope-)semistable vector bundles on $C$ of rank $n$ and determinant isomorphic to $L$ is an irreducible projective variety 
which has been studied intensively for decades. Similar to the Higgs case, the finite group $\Gamma = \mathrm{Pic}^0(C)[n]$ acts on $N_{n,L}$ via tensor product
\begin{equation}\label{action}
\CL \cdot \CE = \CL \otimes \CE, \quad \quad \CL \in \Gamma= \mathrm{Pic}^0(C)[n], \quad \CE \in N_{n,L}.
\end{equation}

Harder and Narasimhan \cite{HN} proved that, when $\mathrm{gcd}(n, d)=1$, the $\Gamma$-action on the cohomology $H^*(N_{n,L}, \BC)$ induced by (\ref{action}) is trivial. Other proofs of the Harder--Narasimhan theorem have been found by Atiyah--Bott \cite{AB} and Hausel--Pauly \cite{HP}.

The following theorem is a generalization of the Harder--Narasimhan theorem for arbitrary degree $d$. It is an immediate consequence of Theorem \ref{thm0.2}.

\begin{thm}\label{thm0.3}
The $\Gamma$-action on $\mathrm{IH}^*(N_{n,L}, \BC)$ induced by (\ref{action}) is trivial. Consequently, we obtain the match of the intersection cohomology groups for the varieties $N_{n,L}$ and $\check{N}_{n,L} := N_{n,L}/\Gamma$:
\begin{equation}\label{SL_PGL}
\mathrm{IH}^*(N_{n,L}, \BC) = \mathrm{IH}^*(\check{N}_{n,L}, \BC). 
\end{equation}
\end{thm}

The varieties $N_{n,L}$ and $\check{N}_{n,L}$ may be viewed as the moduli spaces of semistable $\mathrm{SL}_n$- and $\mathrm{PGL}_n$-bundles on the curve $C$, and Theorem \ref{thm0.3} shows that they share the same intersection cohomology. 

An alternative proof of Theorem \ref{thm0.3} may be obtained by Kirwan's surjectivity for intersection cohomology \cite{Kir0, Kiem}.\footnote{We are grateful to Young-Hoon Kiem and Mirko Mauri for very interesting and helpful discussions on this.} Our approach is to realize Theorem \ref{thm0.3} as a consequence of (a version of) the Hausel--Thaddeus topological mirror symmetry for Hitchin systems. This is close to \cite{HP} in spirit. The proof of Theorem \ref{thm0.3} here suggests that the isomorphism (\ref{SL_PGL}) is essentially a consequence of the fact that the Hitchin systems for $\mathrm{SL}_n$ and $\mathrm{PGL}_n$ share the same Hitchin base over which the decomposition theorems coincide restricting to the generic point. Hence a version of (\ref{SL_PGL}) may hold for general $G$ and its Langlands dual $G^\vee$ which we will explore in subsequent work.

\subsection{Acknowledgements}
 We would like to thank Young-Hoon Kiem and Mirko Mauri for very helpful discussions. We also thank Elsa Maneval for pointing out a missing parity assumption for Theorem \ref{thm0.2} in the previous version, and the referee for careful reading and useful suggestions. J.S. was supported by the NSF grant DMS-2134315.

\section{Support theorems for Hitchin fibrations}\label{sec1}

Throughout the rest of the paper, we fix a curve $C$ of genus $g \geq 2$, an integer $n \geq 2$, and a line bundle $L \in \mathrm{Pic}^d(C)$. Let $D$ be an effective divisor of degree $\mathrm{deg}(D) > 2g-2$.

\subsection{Support theorem}
Assume $n = mr$. Following \cite{MS}, we introduce the endoscopic moduli space $M_{r,L}(\pi)$ associated with a cyclic \'etale Galois cover $\pi: C' \to C$ which plays a crucial role in the cohomological study of $M_{n,L}$.

Let $\pi: C' \to C$ be a degree $m$ cyclic \'etale Galois cover with Galois group $G_\pi \simeq \BZ/m\BZ$. We denote by $M_{r,L}(\pi)$ the moduli of rank $r$ semistable Higgs bundles $(\CE, \theta)$ on $C'$ with respect to the divisor $D' : = \pi^* D$ satisfying that 
\[
\mathrm{det}(\pi_* \CE) \simeq L , \quad  \mathrm{trace}(\pi_*\theta) = 0.
\]
Here $\mathrm{trace}(\pi_*\theta)$ is an element in $H^0(C, \CO_{C}(D))$ which can be viewed as the projection of \[
\mathrm{trace}(\theta) \in H^0(C', \CO_{C'}(D')) = H^0(C, \pi_* \CO_{C'}(D'))\]
to the direct summand component $H^0(C, \CO_C(D))$:
\[
\mathrm{trace}(\pi_*\theta)  \in H^0(C, \CO_C(D)) \subset H^0(C', \pi_*\CO_{C'}(D')). 
\]
The moduli space $M_{r,L}(\pi)$ lies in the moduli of semistable $\mathrm{GL}_r$-Higgs bundles on $C'$, and the Hitchin fibration associated with the latter induces a Hitchin fibration 
\begin{equation}\label{Hitchin_relative}
h_\pi: M_{r,L}(\pi) \to A(\pi);
\end{equation}
see \cite[Section 1.2]{MS} for more details. The Hitchin base $A(\pi)$ naturally sits inside the $\mathrm{GL}_r$-Hitchin base $\widetilde{A}'$ associated with the curve $C'$,
\[
A(\pi)  \subset \widetilde{A}': = \bigoplus_{i=1}^r H^0(C', \CO_{C'}(iD')).
\]
We define the \emph{elliptic locus} $A^\mathrm{ell}(\pi) \subset A(\pi)$ to be the restriction of the elliptic locus of $\widetilde{A}'$ parameterizing integral spectral curves over $C'$.

Our main result of Sections \ref{sec1} and \ref{Sec2} is a support theorem for the Hitchin fibration (\ref{Hitchin_relative}) associated with the endoscopic moduli spaces.

\begin{thm}[Support Theorem]\label{thm1.1}
The generic point of any support of ${Rh_\pi}_* \mathrm{IC}_{M_{r,L}(\pi)}$ lies in the elliptic locus $A^{\mathrm{ell}}(\pi)$.
\end{thm}

When $m =1$ and $\pi = \mathrm{id}$, the moduli space $M_{r,L}(\pi)$ and its Hitchin fibration (\ref{Hitchin_relative}) recover the $\mathrm{SL}_n$-Higgs moduli space $M_{n,L}$ and (\ref{Hitchin}). Hence Theorem \ref{thm1.1} recovers Theorem \ref{thm0.1}. It also generalizes \cite[Theorem 2.3]{MS} for nonsingular ambient spaces. 

Theorem \ref{thm1.1} is a first step towards the study of the global topology for $\mathrm{SL}_n$-Higgs moduli space $M_{n,L}$ and the associated endoscopic moduli spaces. It shows that their global intersection cohomology groups are governed by the (nonsingular) elliptic parts. A similar phenomenon was proven for the $\mathrm{GL}_n$-Higgs moduli spaces and moduli of 1-dimensional semistable sheaves on toric del Pezzo surfaces \cite{MS2}.

\subsection{Weak abelian fibrations}

Since in general the total moduli space $M_{r,L}(\pi)$ may be singular, we use the framework developed in \cite{MS2} to study the Hitchin fibration $h_\pi: M_{r,L}(\pi) \to A(\pi)$. We first show that $h_\pi$ admits the structure as a weak abelian fibration.
 
For a smooth $A(\pi)$-group scheme $g_\pi: P(\pi) \to A(\pi)$ with geometrically connected fibers acting on $M_{r,L}(\pi)$, we say that the triple $(M_{r,L}(\pi), P(\pi), A(\pi))$ is a {weak abelian fibration} of relative dimension $e$, if 
\begin{enumerate}
    \item[(a)] every fiber of the map $g_\pi$ is pure of dimension $e$, and $M_{r,L}(\pi)$ has pure dimension \[\mathrm{dim}M_{r,L}(\pi) =e  +\mathrm{dim}A(\pi),
    \]
    \item[(b)] the action of $P(\pi)$ on $M_{n,L}(\pi)$ has \emph{affine} stabilizers, and
    \item[(c)] the Tate module $T_{\overline{\mathbb{Q}}_l}(P(\pi))$ associated with the group scheme $P(\pi)$ is polarizable.
\end{enumerate}
We refer to \cite[Section 2]{MS2} for more details about these conditions.

In the following, we complete $h_\pi: M_{n,L}(\pi) \to A(\pi)$ into a weak abelian fibration by constructing the group scheme $P(\pi)$ following \cite[Section 4]{dC_SL} and \cite[Section 2.4]{MS}.

Let $\CC \to A(\pi)$ be the universal spectral curve given by the restriction of the universal spectral curve on $\widetilde{A}'$. The relative degree 0 Picard scheme\footnote{It parameterizes line bundles on the closed fibers whose restrictions to each irreducible components are of degree 0. By \cite[Section 8]{Neron} $\mathrm{Pic}^0(\CC/A(\pi))$ an an algebraic space over $A(\pi)$; furthermore, as explained in the last paragraph of \cite[Page 715]{CL} it is indeed a scheme since it sits inside the (quasi-projectve) moduli space of semistable Higgs bundles on $C$.}$\mathrm{Pic}^0(\CC/A(\pi))$ admits a map
\[
\mathrm{Pic}^0(\CC/A(\pi))\to \mathrm{Pic}^0(C)\times A(\pi)
\]
between $A(\pi)$-group schemes as the composition (see the paragraph following \cite[Proposition 2.5]{MS}):
\[
\mathrm{Pic}^0(\CC/A(\pi)) \to \mathrm{Pic}^0(C')\times A(\pi) \to \mathrm{Pic}^0(C)\times A(\pi).
\]
We define $P(\pi)$ to be the identity component of the kernel of this map, which is naturally an $A(\pi)$-group scheme.\footnote{We note that the group scheme $P(\pi)$ is denoted by $P^0$ in \cite{MS}.} By viewing a Higgs bundle in $M_{r,L}(\pi)$ as a pure 1-dimensional semistable sheaf on the spectral curve $C_a$, the $A(\pi)$-group scheme $P(\pi)$ acts on $M_{r,L}(\pi)$ via tensor product (\emph{c.f.} \cite[Lemma 3.4.1]{dCRS}). It was proven in \cite[Proposition 2.6]{MS} that $(M_{r,L}(\pi), A(\pi), P(\pi))$ is a weak abelian fibration of relative dimension $e := \mathrm{dim}M_{r,L}(\pi) - \mathrm{dim}A(\pi)$ when $\mathrm{gcd}(n,d)=1$. In fact, this holds also in the singular case:

\begin{prop}[\emph{c.f. \cite[Proposition 2.6]{MS}}]\label{prop2.2}
The triple $(M_{r,L}(\pi), A(\pi), P(\pi))$ is a weak abelian fibration of relative dimension $e= \mathrm{dim}M_{r,L}(\pi) - \mathrm{dim}A(\pi)$.
\end{prop}

\begin{proof}
The condition (a) is obvious. The condition (c) only concerns the group scheme $P(\pi)$ which was already verified in (ii) of \cite[Proof of Proposition 2.6]{MS}. As in (i) of \cite[Proof of Proposition 2.6]{MS}, the affineness of the stabilizers for the $P(\pi)$-action on $M_{n,L}(\pi)$ follows from the same statement for the corresponding $\mathrm{GL}_r$-Higgs moduli space \cite[Lemma 3.5.4]{dCRS}, since the stabilizers of the $P(\pi)$-actions are closed subgroups of the stabilizers of the $\mathrm{Pic}^0(\CC/\widetilde{A}')$-action. Hence the condition (b) holds as well.
\end{proof}

\subsection{$\delta$-inequalities}
For a closed point $a \in A(\pi)$, we denote by $\delta(a)$ the dimension of the affine part of the algebraic group $P(\pi)_a$ over $a$. This defines an upper semi-continuous function
\[
\delta: A(\pi) \to \BN, \quad a \mapsto \delta(a). 
\]
For a closed subvariety $Z \subset A(\pi)$, we define $\delta_Z$ to be the minimal value of the function $\delta$ on $Z$. Following the strategy of \cite{CL, dC_SL}, it was proven in \cite[Section 2]{MS} that $\delta$-inequalities of the group scheme $P(\pi)$ effectively control the decomposition theorem for $h_\pi: M_{r,L}(\pi) \to A(\pi)$, as we now review.

A key observation of \cite{MS} is that, when $\mathrm{deg}(D)> 2g-2$, a combination of the multi-variable $\delta$-inequality \cite[Proposition 2.7]{MS} and the support inequality (\ref{support}) below implies that the decomposition theorem of $h_\pi: M_{r,L}(\pi) \to A(\pi)$ has no support with generic point lying in $A(\pi) \setminus A^{\mathrm{ell}}(\pi)$.

\begin{prop}[\cite{MS} Section 2.5: Proof of Theorem 2.3 (a)] \label{prop2.3}
Assume that for any support $Z$ of ${Rh_\pi}_*\mathrm{IC}_{M_{r,L}(\pi)}$, we have
\begin{equation}\label{support}
    \mathrm{codim}_{A(\pi)}Z \leq \delta_Z. 
\end{equation}
Then the generic points of all supports are contained in $A^{\mathrm{ell}}(\pi)$.
\end{prop}

When the ambient space $M_{r,L}(\pi)$ is nonsingular, the support inequality (\ref{support}) follows from Ng\^o's work \cite{Ngo}. A singular version was established recently in \cite{MS2} which generalizes Ng\^o's original support inequality.

Recall that $e$ is the relative dimension for the weak abelian fibration $(M_{r,L}(\pi),A(\pi),P(\pi))$ of Proposition \ref{prop2.2}.

\begin{thm}[\cite{MS2} Theorem 1.8]\label{thm2.4}
Suppose we have the vanishing
\begin{equation}\label{truncation}
    \tau_{>2e}\left( Rh_{\pi *}\mathrm{IC}_{M_{r,L}(\pi)}[-\mathrm{dim}M_{r,L}(\pi)]\right) = 0,
\end{equation}
where $\tau_{>\bullet}(-)$ denotes the standard truncation functor.  
Then the inequality (\ref{support}) holds for any support $Z$.
\end{thm}
As a consequence of Proposition \ref{prop2.3} and Theorem \ref{thm2.4}, Theorem \ref{thm1.1} follows from the relative dimension bound (\ref{truncation}), which we prove in the next section.

\section{Proper approximations and support theorems}\label{Sec2}

\subsection{Overview}
The main purpose of this section is to complete the proof of Theorem \ref{thm1.1}. As we explained at the end of Section \ref{sec1}, it suffices to prove the relative dimension bound (\ref{truncation}) which we complete in the following.

\subsection{Proper approximations}

We follow the strategy of \cite[Section 3]{MS2} to prove (\ref{truncation}). 

Let $q: \CW \to W$ be a morphism from a nonsingular Artin stack of finite type to an algebraic variety. Modelled on \cite[Proposition 3.6]{MS2}, we say that $q$ has \emph{a proper approximation} if, for any $R >0$, there exists a nonsingular scheme $W_R$ and an Artin stack $\CX_R$ with a commutative diagram
  \begin{equation}\label{diagram}
    \begin{tikzcd}[column sep=small]
    W_R \arrow[dr, "p_W"] \arrow[rr, hook, "j"] & & \CX_R \arrow[dl, "p_\CX"] \\
       & \CW  & 
\end{tikzcd}
\end{equation}
satisfying the following properties:
\begin{enumerate}
     \item[(a)] $p_\CX$ is an affine space bundle,
     \item[(b)] $j: W_R \hookrightarrow \CX_R$ is an open immersion,
     \item[(c)] the composition $q_R: W_R \xrightarrow{p_W}\CW \xrightarrow{q} W$ is projective, and
    \item[(d)] for the complement $\CZ_R: = \CX_R \setminus W_R$, we have
    \[
    \mathrm{codim}_{\CX_R}\left( \CZ_R\right)>R.
    \]
\end{enumerate}

\begin{prop}\label{Prop3.2}
Assume that $q: \CW \to W$ has a proper approximation. Then the following statements hold.
\begin{enumerate}
    \item[(1)] We have a splitting
    \begin{equation}\label{splitting}
    Rq_* \BC  \simeq \mathrm{IC}_W[-\mathrm{dim}W] \oplus \CK  \in D_c^+(W).
    \end{equation}
    \item[(2)] Let $q': \CW' \to W'$ be the pullback of $q$ along a morphism $f: W' \to W$ with $\CW'$ a nonsingular stack. Then $q'$ has a proper approximation. 
\end{enumerate}
\end{prop}

\begin{proof}
(1) follows from \cite[Section 3.4]{MS2}. In fact, although \cite[Proposition 3.4]{MS2} concerns a more specific geometry, the proof only relies on the diagram (\ref{diagram}) and the properties (a-d) above. More precisely, we view the complex 
\[
Rq_* \BC = R(q \circ p_\CX)_* \BC
\]
as a homotopy colimit of truncations of the direct image complexes $Rq_{R*} \BC$, and use the decomposition theorem for the projective morphism $q_R: W_R \to W$ to deduce the desired splitting (\ref{splitting}). 

(2) is deduced by pulling back the diagram (\ref{diagram}) along $f: W' \to W$. 
\end{proof}

\subsection{Connnecting to $\mathrm{GL}_r$-Hitchin fibrations}

Recall the Hitchin fibration $h_\pi: M_{r,L}(\pi) \to A(\pi)$ associated with $\pi: C' \to C$ with relative dimension
\[
e = \mathrm{dim}M_{r,L}(\pi) - \mathrm{dim}A(\pi).
\]
To verify the relative dimension bound (\ref{truncation}) for $M_{r,L}(\pi)$, we consider the stack $\CM_{r,L}(\pi)$ of semistable Higgs bundles $(\CE, \theta)$ with $\mathrm{det}(\pi_*\CE)\simeq L \in \mathrm{Pic}^d(C)$ and $\mathrm{trace}( \pi_*\theta) = 0$. We denote by $q:\CM_{r,L}(\pi) \to M_{r,L}(\pi)$ the map from the stack to the good moduli space.

For our purpose, we also consider the $\mathrm{GL}_r$-Hitchin fibration $\widetilde{h}: \widetilde{M}'_{r,d} \to \widetilde{A}'$ associated with the curve $C'$. Here $\widetilde{M}'_{r,d}$ is the moduli space of semistable Higgs bundles 
\[
(\CE, \theta), \quad \quad \theta: \CE \to \CE\otimes\CO_{C'}(D'), \quad \quad D' = \pi^*D
\]
of rank $r$ and degree $d$ on $C'$, and $\widetilde{h}$ is the Hitchin fibration sending $(\CE, \theta)$ to its characteristic polynomial 
\[
\mathrm{char}(\theta) \in \widetilde{A}' = \oplus_{i=1}^r H^0(C', \CO_{C'}(iD')). 
\]
We denote by $\widetilde{\CM}'_{r,d}$ the corresponding moduli stack with the natural morphism $\widetilde{q}: \widetilde{\CM}'_{r,d} \to  \widetilde{M}'_{r,d}$. We recall the following proposition from \cite{MS2} concerning $\widetilde{\CM}'_{r,d}$.

\begin{prop}[\cite{MS2} Proposition 2.9 (2) and Proposition 3.6] \label{Prop3.3}
The stack $\widetilde{\CM}'_{r,d}$ is nonsingular, and $\widetilde{q}: \widetilde{\CM}'_{r,d} \to \widetilde{M}'_{r,d}$ has a proper approximation.
\end{prop}

Now we connect the moduli spaces and stacks for the endoscopic groups and $\mathrm{GL}_r$ via the construction of \cite[Section 5]{MS}.

We consider the moduli space $\widetilde{M}_{1,0}$ (resp. moduli stack $\widetilde{\CM}_{1,0}$) of Higgs bundles on $C$ with rank 1 and degree 0. More concretely, they can be described as:
\[
\widetilde{M}_{1,0} = \mathrm{Pic}^0(C)\times H^0(C, \CO_C(D)), \quad \widetilde{\CM}_{1,0} = {\CP}ic^0(C)\times H^0(C, \CO_C(D))
\]
where $\mathrm{Pic}^0(-)$ and ${\CP}ic^0(-)$ stand for the degree 0 Picard scheme and stack respectively. We denote by 
\[
q_P: \widetilde{\CM}_{1,0} \to \widetilde{M}_{1,0}
\] 
the natural morphism. The group scheme $ \widetilde{M}_{1,0}$ acts on $\widetilde{M}'_{r,d}$:
\[
(\CL, \sigma) \cdot (\CE, \theta) = (\pi^*\CL \otimes \CE, \pi^*\sigma + \theta), \quad \quad (\CL, \sigma) \in \widetilde{M}_{1,0}, \quad (\CE, \theta) \in \widetilde{M}'_{r,d}
\]
which induces a morphism
\[
t: \widetilde{M}_{1,0} \times M_{r,L}(\pi) \to \widetilde{M}'_{r,d}
\]
by restricting the action to $M_{r,L}(\pi) \subset \widetilde{M}'_{r,d}$. The map $t$ can be interpreted as the quotient map by the finite group $\Gamma = \mathrm{Pic}^0(C)[n]$ acting diagonally on the two factors; see \cite[Section 5.3]{MS}. Similarly, we have the $\Gamma$-quotient map for the moduli stacks: 
\[
 \widetilde{\CM}_{1,0} \times \CM_{r,L}(\pi) \to \widetilde{\CM}'_{r,d}
\]
inducing the following Cartesian diagram
\begin{equation}\label{BC}
\begin{tikzcd}
\widetilde{\CM}_{1,0} \times \CM_{r,L}(\pi) \arrow[r] \arrow[d]
& \widetilde{\CM}'_{r,d} \arrow[d] \\
\widetilde{M}_{1,0} \times M_{r,L}(\pi) \arrow[r, "t"]
& \widetilde{M}'_{r,d}
\end{tikzcd}
\end{equation}
where the horizontal arrows are quotient maps by the $\Gamma$-actions and the vertical arrows are the maps from the stacks to the good moduli spaces.

\begin{prop}\label{Prop3.4}
The moduli stack $\CM_{r,L}(\pi)$ is nonsingular, and the left vertical map of (\ref{BC})
\[
g:= q_P \times q: \widetilde{\CM}_{1,0} \times \CM_{r,L}(\pi)  \to \widetilde{M}_{1,0} \times M_{r,L}(\pi) 
\]
has a proper approximation.
\end{prop}

\begin{proof}
By the discussion in the proof of \cite[Proposition 4.1]{MS}, the obstruction space for an element $(\CE, \theta) \in \CM_{r,L}(\pi)$ is the second cohomology group of the following complex
\[
 \left[(\pi_*\CE{nd}(\CE))_0 \xrightarrow{\pi_*\mathrm{ad}(\theta)} (\pi_*\CE{nd}(\CE))_0\otimes \CO_C(D)\right]
\]
obtained by removing the trace from the pushforward of the complex
\begin{equation}\label{obstruction}
\left[\CE{nd}(\CE) \xrightarrow{\mathrm{ad}(\theta)} \CE{nd}(\CE)\otimes \CO_{C'}(D')\right].
\end{equation}
Here $(\pi_*\CE{nd}(\CE))_0$ denotes the kernel with respect to the trace on the curve $C$:
\[
\mathrm{tr}_C: \pi_* \CE{nd}(\CE) \xrightarrow{\pi_*\mathrm{tr}_{C'}} \pi_*\CO_{C'} \to \CO_C
\]
In particular, the obstruction space for $(\CE, \theta) \in \CM_{r,L}(\pi)$ is a subspace of the second cohomology group of (\ref{obstruction}) on $C'$ which is actually the obstruction space for $(\CE, \theta) \in \widetilde{\CM}'_{r,d}$ by viewing $(\CE, \theta)$ as a $\mathrm{GL}_r$-Higgs bundle on $C'$. Its vanishing follows from the (the proof of) Proposition \ref{Prop3.3} on the smoothness of $\widetilde{\CM}'_{r,d}$. This shows that $\CM_{r,L}(\pi)$ is nonsingular.

Consequently, we obtain the smoothness of $\widetilde{\CM}_{1,0} \times \CM_{r,L}(\pi)$. The second part is a corollary of Proposition \ref{Prop3.2} (2) and Proposition \ref{Prop3.3}.
\end{proof}

By Propositions \ref{Prop3.2} (1) and \ref{Prop3.4}, we get the following result.

\begin{cor}\label{Cor3.5}
We have a splitting
\begin{equation}\label{eqn15}
Rg_* \BC \simeq \mathrm{IC}_{\widetilde{M}_{1,0} \times M_{r,L}(\pi) }[-\mathrm{dim}\widetilde{M}_{1,0} -\mathrm{dim}M_{r,L}(\pi)]  \oplus \CK
\end{equation}
with $\CK$ some complex bounded from below.
\end{cor}

\subsection{Proof of Theorem \ref{thm1.1}}\label{Sec2.4}
We verify (\ref{truncation}) in this section which completes the proof of Theorem \ref{thm1.1}. For convenience, we use the following simplified notation (only) in Section \ref{Sec2.4}:
\[
\begin{split}
    H:= \widetilde{M}_{1,0} ,\quad  M:= M_{r,L}(\pi),\quad  \widetilde{M}':= 
\widetilde{M}'_{r,d}, \\ \CH:= \widetilde{\CM}_{1,0}, \quad \CM:= \CM_{r,L}(\pi), \quad  \widetilde{\CM}':= \widetilde{\CM}'_{r,d}.
\end{split}
\]

\medskip
\noindent {\bf Fact 1.} For the morphism $q: \CM \to M$, we have a splitting
\[
Rq_* \BC \simeq \mathrm{IC}_M[-\dim M] \oplus \CK' .
\]

\begin{proof}[Proof of Fact 1.]
Since $H$ is nonsingular, we have
\[
\mathrm{IC}_{H\times M} \simeq \BC_H[\mathrm{dim}H] \boxtimes \mathrm{IC}_M.
\]
On the other hand, the lefthand side of (\ref{eqn15}) is equal to
\[
Rg_* \BC  = \bigoplus_{i\geq 0}\BC_H \boxtimes Rq_* \BC_{M} [-2i]. 
\]
Hence by restricting (\ref{eqn15}) to $\mathrm{pt}\times M \subset H \times M$, we obtain that
\[
\bigoplus_{i\geq 0 } Rq_* \BC_M [-2i]\simeq \mathrm{IC}_M[-\mathrm{dim}M] \oplus \cdots \in D^+_c(M).
\]
Since $\mathrm{IC}_M[-\mathrm{dim}M]$ is simple, it has to be a direct summand component of some $Rq_*\BC_M[-2k]$. By comparing over the nonsingular locus of $M$, we see that $k=0$. 
\end{proof}

\medskip
\noindent {\bf Fact 2.} Let $h_\CM: \CM \to A(\pi)$ be the composition 
\[
h_\CM: \CM \xrightarrow{q} M \xrightarrow{h_\pi} A(\pi).
\]
Then we have 
\[
\tau_{>2e}\left(Rh_{\CM!} \BC_\CM\right) =0, \quad \quad e = \mathrm{dim}M-\mathrm{dim}A(\pi) = \mathrm{dim}\CM-\mathrm{dim}A(\pi)+1.
\]
\begin{proof}[Proof of Fact 2]

We consider the map $h_{\widetilde{\CM}'} :  \widetilde{\CM}' \to \widetilde{A}'$ given as the composition
\[
h_{\widetilde{\CM}'}= \widetilde{h}\circ \widetilde{q}: \widetilde{\CM}' \to \widetilde{M}' \to \widetilde{A}'.
\]

By \cite[Proposition 2.9 (1)]{MS2} (see also \cite[Section 10]{CL}) we have the dimension bound for any closed fiber:
\[
\mathrm{dim}h^{-1}_{\widetilde{\CM}'}(a) \leq   \mathrm{dim}\widetilde{\CM}'-\mathrm{dim}\widetilde{A}' = e +(g-1), \quad \forall a\in \widetilde{A}'.
\]
Hence, for the morphism $h_{\CH\times\CM}: \CH \times \CM \to H^0(C, \CO_C(D)) \times A(\pi)$ given by the composition
\[
h_{\CH\times\CM}: \CH \times \CM \to H \times M \to H^0(C, \CO_C(D)) \times A(\pi), 
\]
we obtain from the diagram (\ref{BC}) that
\[
\mathrm{dim}h_{\CH\times\CM}^{-1}(w,s)   =   \mathrm{dim}h^{-1}_{\widetilde{\CM}'}\left(t(w,s)\right) \leq   e+ (g-1),  \quad \quad \forall (w,s)\in H^0(C, \CO_C(D))\times A(\pi).
\]
On the other hand,
\[
\mathrm{dim}h_{\CH\times\CM}^{-1}(t,s) = \mathrm{dim}h_\CM^{-1}(s) +(g-1).
\]
Consequently $\mathrm{dim}h_\CM^{-1}(s) \leq e$ for any closed point $s\in A(\pi)$. Fact 2 follows from \cite[Lemma 3.5]{MS2} and base change.
\end{proof}

As explained in the paragraph following \cite[Proposition 3.4]{MS}, Facts 1 and 2 imply the relative dimension bound (\ref{truncation}) immediately. This completes the proof of Theorem \ref{thm1.1}. \qed

\section{The Hausel--Thaddeus conjecture}

\subsection{Overview}
We complete the proof of Theorem \ref{thm0.2} in this section. As a consequence of Theorem \ref{thm1.1}, we first show that both sides of (\ref{thm0.2_a}) are semisimple objects with $A_\gamma$ as the only support. Then Theorem \ref{thm0.2} (a) is reduced to showing the desired isomorphism over an arbitrary Zariski open subset of the locus $A_\gamma \subset A$. This is essentially identical to the proof of \cite[Theorem 3.2]{MS} which only relies on the calculation over the elliptic locus \cite{Ngo, Yun3}.

Theorem \ref{thm0.2} (b) is more complicated, since this is a new phenomenon when $\mathrm{gcd}(n,d)\neq 1$.\footnote{When $\mathrm{gcd}(n,d) = \mathrm{gcd}(n,d')=1$, the condition (\ref{condition}) specializes to the condition that $\kappa' = d'^{-1}d\kappa$ as in \cite[Theorem 0.5]{MS}.} Again, we use the support theorem to reduce the desired isomorphism to a calculation of the $G_\pi$-action on the $m$ components of the moduli space $M_{r,L}(\pi)$. This is carried out in Section \ref{Sec3.5}.

In Section \ref{HT_conj}, we further discuss the connection between Theorem \ref{thm0.2} and the original formulation of the Hausel--Thaddeus conjecture \cite{HT}.

\subsection{Supports for $h: M_{n,L} \to A$}
Recall the $\mathrm{SL}_n$-Hitchin fibration $h: M_{n,L} \to A$, and the elliptic locus $A^{\mathrm{ell}} \subset A$ which is the open subset of $A$ consisting of integral spectral curves. The fiberwise $\Gamma$-action on $M_{n,L}$ yields the canonical decomposition 
\begin{equation*}
{Rh}_* \mathrm{IC}_{M_{n,L}}  = \bigoplus_{\kappa} \left({Rh}_* \mathrm{IC}_{M_{n,L}}\right)_\kappa , \quad \kappa \in \hat{\Gamma}.
\end{equation*}
Let $\gamma \in \Gamma$ be the element matched with the nontrivial character $\kappa  \in \hat{\Gamma}$ via the Weil pairing (\ref{Weil_Pairing}). Ng\^o proved in \cite[Theorem 7.8.5]{Ngo} that the restriction of the object 
\begin{equation}\label{kappa_object}
\left({Rh}_*\mathrm{IC}_{M_{n,L}}\right)_\kappa 
\end{equation}
to $A^{\mathrm{ell}}$ has 
\[
A_\gamma^{\mathrm{ell}}: = A_\gamma \cap A^{\mathrm{ell}}  \subset A
\]
as its only support. Hence we obtain the following proposition concerning the lefthand side of (\ref{thm0.2_a}) from Theorem \ref{thm0.1}:

\begin{prop}\label{prop3.1}
We have that $A_\gamma$ is the only support of the object (\ref{kappa_object}).
\end{prop}

\subsection{The moduli spaces $M_{r,L}(\pi)$ and $M^\gamma_{n,L}$}
Now we prove a support theorem for the fibration $h_\gamma: M^\gamma_{n,L} \to A_\gamma$ concerning the object in the righthand side of (\ref{thm0.2_a}). We achieve this using the moduli space $M_{r,L}(\pi)$ discussed in Sections \ref{sec1} and \ref{Sec2}.

Assume $\kappa$ has order $m$ in $\hat{\Gamma}$. Therefore $\gamma$ is an $m$-torsion line bundle. Let $\pi: C' \to C$ be the degree $m$ cyclic \'etale Galois cover associated with $\gamma$ \cite[Section 1.3]{MS}. In the following, we construct the commutative diagram
\begin{equation}\label{diagram111}
\begin{tikzcd}
M_{r,L}(\pi) \arrow[r, "q_M"] \arrow[d, "h_{\pi}"]
& M_\gamma \arrow[d, "h_\gamma"] & \\
A{(\pi)} \arrow[r, "q_A"]
& A_\gamma 
\end{tikzcd}
\end{equation}
connecting $h_\pi$ and $h_\gamma$, where the bottom horizontal map $q_A$ is the $G_\pi$-quotient; see \cite[Section 1.5]{MS} for the coprime case. Note that the map $q_M$ is the free $G_\pi$-quotient in the coprime case, but it is more complicated in general without the coprime assumption (Remark \ref{rmk1}).

We first review the construction of \cite[Section 7]{HT} which gives the top horizontal map $q_M$. Let $(\CE, \theta)$ be a rank $r$ Higgs bundle on the curve $C'$, then $(\pi_* \CE, \pi_*\theta)$ is a rank $n (= rm)$ Higgs bundle on $C$. Here the bundle $\pi_*\CE$ is simply the pushforward of $\CE$ along $\pi: C' \to C$, and the Higgs field $\theta$ is given by descending the block-diagonal Higgs field $\bigoplus_{g\in G_\pi} g^*\theta$ on the vector bundle
\begin{equation}\label{eqn17}
\pi^*\pi_*\CE. = \bigoplus_{g\in G_\pi} g^*\CE
\end{equation}
along the $G_\pi$-quotient $\pi: C'\to C$. We recall the following well-known lemma.

\begin{lem}\label{lem3.2}  The Higgs bundle $(\CE, \theta)$ is semistable if and only if $(\pi_*\CE , \pi_*\theta)$ is semistable.
\end{lem}

\begin{proof}
The \emph{if} part is obvious: for any sub-Higgs bundle destabilizing $(\CE, \theta)$, its pushforward along $\pi$ will destabilize $(\pi_*\CE, \pi_*\theta)$. For the \emph{only if} part, we consider the decomposition (\ref{eqn17}):
\begin{equation}\label{eqn18}
\pi^*\pi_*(\CE, \theta) = \bigoplus_{g\in G_\pi} g^*(\CE, \theta).
\end{equation}
In particular, if $(\CE,\theta)$ is semistable, then (\ref{eqn18}) as a direct summand of semistable Higgs bundles of the same slope is also semistable. Hence the pullback of any sub-Higgs bundle destabilizing $(\pi_*\CE, \pi_*\theta)$ will destabilize (\ref{eqn18}) as well. This completes the proof.
\end{proof}

By Lemma \ref{lem3.2}, the push forward $\pi_*$ induces a morphism between the moduli spaces
\begin{equation}\label{eqn19}
M_{r,L}(\pi) \to M_{n,L}.
\end{equation}
Moreover, by \cite[Proposition 3.3]{NR}, the restriction of (\ref{eqn19}) to the Zariski dense open subset $M_{r,L}(\pi)^\circ \subset M_{r,L}(\pi)$ formed by points not fixed by any element of $G_\pi$ is a free $G_\pi$-quotient with image lying in $M^\gamma_{n,L}$. In conclusion, we obtain 
\[
q_M: M_{r,L}(\pi) \to M^\gamma_{n,L} \subset M_{n,L}.
\]
which completes the diagram (\ref{diagram111}).

\begin{rmk}\label{rmk1}
When $\mathrm{gcd}(n,d)=1$ so that there is no strictly semistable objects, both varieties $M_{r,L}(\pi)$ and $M^\gamma_{n,L}$ are nonsingular, and the map $q_M$ induced by $\pi_*$ is a \emph{free} $G_\pi$-quotient \cite[Proposition 7.1]{HT}. However, this may fail when $\mathrm{gcd}(n,d)\neq 1$. For example, the rank 1 stable Higgs bundle $(\CO_{C'}, 0)$ is a $G_\pi$-fixed point.
\end{rmk}

\begin{lem}\label{lem3.4}
We have a splitting
\[
{Rq_M}_* \mathrm{IC}_{M_{r,L}(\pi)} = \mathrm{IC}_{M^\gamma_{n,L}} \oplus \cdots.
\]
\end{lem}

\begin{proof}
Over an open subset of $M^\gamma_{n,L}$ where $q_M$ is a free $G_\pi$-quotient, we have the canonical splitting
\[
{Rq_M}_* \BC = \left({Rq_M}_* \BC\right)^{G_\pi} \oplus \left({Rq_M}_* \BC\right)_{\mathrm{var}} = \BC \oplus \left({Rq_M}_* \BC\right)_{\mathrm{var}}
\]
with $\left({Rq_M}_* \BC\right)_{\mathrm{var}}$ the variant part. The lemma follows.
\end{proof}

To analyze the supports for $h_\gamma: M^\gamma_{n,L} \to A_\gamma$, we note the following standard lemma.

\begin{lem}\label{lem3.5}
Let $f: X \to Y$ be a finite surjective map between irreducible varieties. Then for any semisimple perverse sheaf $\mathrm{IC}_X(\CL)$ with full support $X$, the pushforward $f_*\mathrm{IC}_X(\CL)$ is s semisimple perverse sheaf with full support $Y$.
\end{lem}

\begin{proof}
To show that $f_*\mathrm{IC}_X(\CL)$ is an intermediate extension of a local system on an open subset of $Y$, it suffices to prove the support condition (see \cite[Section 2.1 (12),(13)]{dCM1}):
\[
\mathrm{dim}\left(\mathrm{supp}(\CH^{-i}(-)\right) <i, \quad \quad \textup{for } i<\mathrm{dim}Y
\]
for $f_*\mathrm{IC}_X(\CL)$ and its dual. This follows from the finiteness of $f$ and the same support conditions for $\mathrm{IC}_X(\CL)$ and its dual on $X$. 
\end{proof}

\begin{prop}\label{prop3.6}
Assume that $\gamma\in \Gamma$ and $\kappa \in \hat{\Gamma}$ are matched via the Weil pairing (\ref{Weil_Pairing}), and $\kappa' \in \langle \kappa \rangle$. The object 
\begin{equation}\label{ttt}
\left({Rh_\gamma}_* \mathrm{IC}_{M^\gamma_{n,L}}\right)_{\kappa'}
\end{equation}
has full support $A_\gamma$.
\end{prop}

\begin{proof}
We first consider the map $h_\pi: M_{r,L}(\pi) \to A(\pi)$ and observe that the object
\begin{equation}\label{eqn20}
\left({Rh_\pi}_* \mathrm{IC}_{M_{r,L}(\pi)}\right)_{\kappa'}
\end{equation}
has full support $A(\pi)$ for $\gamma$ and $\kappa'$ as in the assumption and $\pi: C' \to C$ given by $\gamma$. When $\mathrm{gcd}(n,d) =1$ this is verified in \cite[Theorem 2.3 (b) and Proposition 2.10]{MS}, which relies on the support theorem (\cite[Theorem 2.3 (a)]{MS}) and a direct calculation over the elliptic locus. Since the moduli space $M_{r,L}(\pi)$ is nonsingular restricting over the elliptic locus and the calculation of \cite{MS} over the elliptic locus does not rely on the coprime assumption, we obtain that the full support property still holds for (\ref{eqn20}) as a consequence of Theorem \ref{thm1.1}.

To prove the proposition, we use the commutative diagram (\ref{diagram111}) which induces a canonical $\Gamma$-equivariant isomorphism 
\[
{Rq_A}_*{Rh_\pi}_* \mathrm{IC}_{M_{r,L}(\pi)} = {Rh_\gamma}_*{Rq_M}_* \mathrm{IC}_{M_{r,L}(\pi)}.
\]
Taking the $\kappa'$-isotypic parts, we get
\begin{equation}\label{eqn21}
{Rq_A}_*\left({Rh_\pi}_* \mathrm{IC}_{M_{r,L}(\pi)}\right)_{\kappa'} = \left({Rh_\gamma}_*{Rq_M}_* \mathrm{IC}_{M_{r,L}(\pi)}\right)_{\kappa'}
\end{equation}
where both sides are semisimple objects due to the decomposition theorem. Since $q_A$ is a finite quotient map and (\ref{eqn20}) has full support $A(\pi)$, the lefthand side (\ref{eqn21}) has full support $A_\gamma$ by Lemma \ref{lem3.5}. Furthermore, Lemma \ref{lem3.4} implies that (\ref{ttt}) is a direct summand component of the righthand side of (\ref{eqn21}). This completes the proof.
\end{proof}

\subsection{Proof of Theorem \ref{thm0.2} (a)}
Theorem \ref{thm0.2} (a) is an immediate consequence of Propositions \ref{prop3.1} and \ref{prop3.6}.

More precisely, since both sides of (\ref{thm0.2_a}) have $A_\gamma$ as their only supports, it suffices to show the isomorphism over an \emph{arbitrary} open subset of $A_\gamma$ which is proven essentially by \cite[Theorem B]{Yun3}; see also \cite[Theorem 3.2]{MS} . We note that in the proof of \cite[Theorem 3.2]{MS} the parity assumption on $\mathrm{deg}(D)$ is needed in order to apply the endoscopic correspondence. \qed

\begin{rmk}
In fact, even without the coprime assumption, the proof of \cite[Theorem 3.2]{MS} works over the elliptic locus $A^\mathrm{ell}_\gamma \subset A_\gamma$. In particular, we may choose the open subset in the proof above to be the elliptic locus.
\end{rmk}

\subsection{Proof of Theorem \ref{thm0.2} (b)}\label{Sec3.5}



Since the object (\ref{ttt}) has full support $A_\gamma$, its isomorphism class is determined by the restriction over a Zariski open subset. In view of the diagram (\ref{diagram111}), it suffices to treat the $G_\pi$-equivariant objects
\begin{equation}\label{eqn23}
\left({Rh_\pi}_*\BC_{h_\pi^{-1}(V)}\right)_{\kappa'}
\end{equation}
over an arbitrary Zariski open $V \subset A(\pi)$. After shrinking $V$, we may assume that all the fibers of $h_\pi$ are nonsingular and $G_\pi$ acts freely on $V$. By \cite[Proposition 7.2.3]{Ngo} (see \cite[Theorem 5.0.2]{dCRS} for the Hodge module version), the isomorphism class of the object (\ref{eqn23}) is completely determined by the $G_\pi$-equivariant local system given by the relative top degree cohomology:
\[
\left({R^{2s}h_\pi}_*\BC_{h_\pi^{-1}(V)}\right)_{\kappa'}.
\]
Here $s$ is the dimension of a fiber of $h_\pi$ over $V$. The sheaf 
\[
{R^{2s}h_\pi}_*\BC_{h_\pi^{-1}(V)}
\]
is a rank $m$ trivial local system indexed by the $m$ connected components of a general fiber of $h_{\pi}$, which are further identified with the $m$ connected components of the degree $d$ Prym variety 
\[
\mathrm{Prym}^d(C'/C):= \mathrm{Nm}^{-1}(L), \quad \mathrm{Nm}: \mathrm{Pic}^d(C') \to \mathrm{Pic}^d(C) 
\]
associated with the cyclic Galois cover $\pi: C' \to C$; see \cite[Section 1]{MS}.

In conclusion, the isomorphism class of (\ref{eqn23}) is completely determined by the $G_\pi$- and the $\Gamma$-actions on the $m$ connected components of $\mathrm{Prym}^d(C'/C)$. These two actions commute with each other. 

Now we want to connect the Hitchin fibrations \[
h_{\pi,L}: M_{r,L}(\pi) \to A(\pi), \quad\quad h_{\pi,L'}: M_{r,L'}(\pi) \to A(\pi)
\]
where the line bundles $L$ and $L'$ are of degrees $d$ and $d'$ respectively.\footnote{In this section we use $h_{\pi,L}$ to denote the Hitchin fibration $M_{r,L}(\pi) \to A(\pi)$ to indicate its dependence on the line bundle $L$.} 

We first note the following elementary lemma which justifies the condition (\ref{condition}).

\begin{lem}\label{lem3.8}
There is an integer $q$ coprime to $n$ such that 
\[
d = d'q ~~\mathrm{mod}~~n.
\]
\end{lem}
\begin{proof}
Assume that 
\[
\mathrm{gcd}(n,d) = \mathrm{gcd}(n,d') = a.
\]
Then both the primary ideals $(d)$ and $(d')$ of $\BZ/n\BZ$ coincide with $(a)$. Hence the generators $d$ and $d'$ differ by a unit of $\BZ/n\BZ$.
\end{proof}

In the following, the integer $q$ will be chosen as in Lemma \ref{lem3.8}. The proof of Theorem \ref{thm0.2} (b) follows from the following two steps.

\subsubsection{Step 1: Connecting $h_{\pi,L'}$ to $h_{\pi,L'^{\otimes q}}$}\label{3.5.1} 
Since the $G_\pi$-equivariant objects (\ref{eqn23}) associated with the Hitchin fibrations $h_{\pi,L'}$ and $h_{\pi,L'^{\otimes q}}$ are completely determined by the $G_\pi$- and the $\Gamma$-actons on the Prym varieties 
\[
\mathrm{Prym}^{d'}(C'/C): = \mathrm{Nm}^{-1}(L')
\]
and
\[
\mathrm{Prym}^{d'q}(C'/C) : = \mathrm{Nm}^{-1}(L'^{\otimes q})
\]
respectively. An identical argument as for \cite[Proposition 2.11]{MS} yields
\begin{equation*}
\left(\mathrm{Rh_{\pi,L'}}_* \BBC_{h_{\pi,L'}^{-1}(V)} \right)_{q\kappa} \simeq \left(\mathrm{Rh_{\pi,L'^{\otimes q}}}_* \BBC_{h^{-1}_{\pi,L'^{\otimes q}}(V)} \right)_{\kappa} \in D_c^b(V).
\end{equation*}
In view of Proposition \ref{prop3.6}, this further implies that 
\begin{equation}\label{extra}
\left({Rh_\gamma}_*\mathrm{IC}_{M^\gamma_{n,L'}}\right)_{q\kappa} \simeq  \left({Rh_\gamma}_* \mathrm{IC}_{M^\gamma_{n,L'^{\otimes q}}}\right)_{\kappa}  \in  D^b_c(A_\gamma).
\end{equation}

\subsubsection{Step 2: Connecting $M^\gamma_{n,L'^{\otimes q}}$ and $M^\gamma_{n,L}$} 
By the choice of $q$ we have
\begin{equation}\label{eqn25}
\mathrm{deg}(L'^{\otimes q}) - \mathrm{deg}(L) = 0 ~~ \mathrm{mod}~~ n.
\end{equation}
Note that for two line bundles $L_1$ and $L_2$ with $L_1 = L_2 \otimes N^{\otimes n}$, there is a natural identification of the moduli spaces 
\begin{equation*}
    M_{n, L_1} \xrightarrow{\simeq} M_{n, L_2}, \quad (\CE, \theta) \mapsto (\CE\otimes N, \theta)
\end{equation*}
compatible with the $\Gamma$-actions and the Hitchin fibrations. Therefore, by (\ref{eqn25}) we have natural isomorphisms
\[
M_{n,L'^{\otimes q}} \xrightarrow{\simeq} M_{n,L},\quad \quad M^\gamma_{n,L'^{\otimes q}} \xrightarrow{\simeq} M^\gamma_{n,L}
\]
which further induce
\begin{equation}\label{extra1}
    \left({Rh_\gamma}_* \mathrm{IC}_{M^\gamma_{n,L'^{\otimes q}}}\right)_{\kappa} \simeq  \left({Rh_\gamma}_* \mathrm{IC}_{M^\gamma_{n,L}}\right)_{\kappa}.
\end{equation}

The proof of Theorem \ref{thm0.2} (b) is completed by combining (\ref{extra}) and (\ref{extra1}).
\qed

\subsection{The Hausel--Thaddeus conjecture}\label{HT_conj}

In this section, we give a few remarks regarding the relation of our result with the Hausel--Thaddeus conjecture.

The original form of the Hausel--Thaddeus conjecture involves 
Higgs bundles of type $\mathrm{SL}_n$ and $\mathrm{PGL}_n$ with $D = K_C$ and in the coprime setting $\mathrm{gcd}(n,d) = 1$. It relates the singular cohomology of $M_{n,L}$ with the stringy cohomology of $[M_{n,L}/\Gamma]$, twisted by a particular gerbe $\alpha$
whose appearance is motivated by SYZ mirror symmetry.  In the coprime setting, as explained in the appendix of \cite{LW}, the $\alpha$-twisted cohomology of the sector 
\[
[M_{n,L}^\gamma/\Gamma], \quad  \gamma \in \Gamma
\]
is equivalent to a certain isotypic component of the singular cohomology of $M_{n,L}^\gamma$. Hence the original Hausel--Thaddeus formulation is implied by the formulation as in Theorem \ref{thm0.2}, after passing to global cohomology. In the non-coprime setting, however, it is not clear to us how to define the corresponding gerbe $\alpha$ on the singular stack $[M_{n,L}/\Gamma]$ and so we do not have a direct definition of the $\alpha$-twisted intersection cohomology.  As a result, the formulation we give here in terms of the endoscopic decomposition seems more natural.

If we consider the case of Higgs bundles with $D= K_C$ but general degree $d$, then our argument no longer applies; contrary to Theorems \ref{0.1} and \ref{thm1.1} the decomposition theorem for the Hitchin fibration have many additional supports outside the elliptic locus (\emph{c.f. \cite{dCHeM}}). When $\mathrm{gcd}(n,d)=1$, we deduce in \cite{MS} the Hausel-Thaddeus conjecture for $D=K_C$ from the cases of $\mathrm{deg}(D)>2g-2$ using vanishing cycle techniques. However, the approach of \cite{MS} cannot be applied directly to deduce Theorem \ref{thm0.2} (as conjectured by Mauri \cite{Mauri} in the degree 0 case) for $D=K_C$ when $\mathrm{gcd(n,d)}\neq 1$. More precisely, the main ingredient of \cite{MS} is Theorem 4.5 \emph{loc. cit}, which relies on the smoothness of the evaluation map of Proposition 4.1 \emph{loc. cit}. The smoothness fails when there are strictly semistable points.


From the perspective of enumerative geometry, another natural option is to work with the cohomology of the so-called BPS sheaf $\phi_{\mathrm{BPS}}$, a perverse sheaf on $M_{n,L}$ defined by Davison--Meinhardt \cite{DM1} and Toda \cite{Toda}.  When $\deg(D) > 2g-2$, the BPS-cohomology coincides with intersection cohomology but for $D=K_C$ these two are different. Note that combining the recent work \cite{KM} and Theorem \ref{thm0.2} may privide a proof of a version of the Hausel--Thaddeus conjecture for the BPS-cohomology for $D = K_C$; the approach of Davison \cite{DD} further suggests a path line to deduce the $D = K_C$ case of Theorem \ref{thm0.2} from the BPS-cohomology.


Finally, it is reasonable to expect Theorem \ref{thm0.2} can be extended to the case of Higgs bundles for a general reductive group $G$ and its Langlands dual $G^\vee$, and we hope to explore this in subsequent work.

\section{Vector bundles and Higgs bundles}

In this section, we discuss the interplay between the moduli of vector bundles and the moduli of Higgs bundles, and complete the proof of Theorem \ref{thm0.3}. As before, we fix a line bundle $L \in \mathrm{Pic}^d(C)$ and an effective divisor $D$ with $\mathrm{deg}(D)$ even and greater than $2g-2$.


\subsection{Moduli spaces $M_{n,L}$ and $N_{n,L}$}
We would like to study the topology of $N_{n,L}$ via the Higgs moduli space $M_{n,L}$. 

We consider the $\BC^*$-action on $M_{n,L}$ by the scaling action on the Higgs field:
\[
\lambda\cdot (\CE, \theta) = (\CE, \lambda \theta), \quad \quad \lambda \in \BC^*.
\]
The $\BC^*$-fixed locus $F \subset M_{n,L}$ can be decomposed as
\[
F = N_{n,L} \sqcup  F'.
\]
Here the first connected component parameterizes (S-equivalence classes of) semistable Higgs bundles with $\theta =0$ which is naturally isomorphic to $N_{n,L}$. The restriction of the $\Gamma$-action on $M_{n,L}$ to $N_{n,L}$ recovers (\ref{action}). 

We apply hyperbolic localization to connect the intersection cohomology of the moduli spaces $M_{n,L}$ and $N_{n,L}$.

\subsection{Hyperbolic Localization}
We consider the following subvarieties of $M_{n,L}$ obtained from the scaling $\BC^*$-action:
\[
M^+: = \{x\in M_{n,L}: \lim_{\lambda \to 0 }\lambda\cdot x \in F\}, \quad  M^-: = \{x\in M_{n,L}: \lim_{\lambda\to \infty}\lambda\cdot x \in F\}.
\]
Let $f^{+}, f^{-}, g^+, g^-$ be the inclusions
\begin{equation}\label{fg}
f^+: F \hookrightarrow M^+, \quad  f^-: F \hookrightarrow M^-, \quad g^+: M^+ \hookrightarrow M_{n,L}, \quad  g^-: M^- \hookrightarrow M_{n,L}.
\end{equation}
Following \cite{Kir2, Braden}, we consider the \emph{hyperbolic localization functor}:
\begin{equation}\label{functor}
(-)^{!*} : D^b_c(M_{n,L}) \rightarrow D^b_c(F), \quad \CK \mapsto (f^+)^*(g^+)^!\CK.
\end{equation}
We obtain from the main theorem of Kirwan \cite{Kir2} that there is an isomorphism
\begin{equation}\label{localization}
\mathrm{IH}^*(M_{n,L}, \BC) \simeq  H^*\left(F,~~ (\mathrm{IC}_{M_{n,L}})^{!*}[-\mathrm{dim}M_{n,L}]\right).
\end{equation}
In fact, Kirwan proved (\ref{localization}) for normal projective varieties with $\BC^*$-actions. In the case of the moduli of Higgs bundles, one may deduce (\ref{localization}) by applying Kirwan's theorem to a compactification $M_{n,L} \subset \overline{M}_{n,L}$ \cite{Compact, Hausel} where the $\BC^*$-action can be lifted, and then restrict the isomorphism (\ref{localization}) for $\overline {M}_{n,L}$ to the open subvariety $M_{n,L}$; see the first paragraph in \cite[Proof of Corollary 1.5]{HP}. 

Concerning the righthand side of (\ref{localization}), Braden showed in \cite{Braden} that there is a splitting
\begin{equation}\label{decomp}
(\mathrm{IC}_{M_{n,L}})^{!*} \simeq \bigoplus_{i} \mathrm{IC}_{Y_i}(\CL_i)[d_i]
\end{equation}
with $Y_i \subset F$ irreducible closed subvarieties, $\CL_i$ local systems on open subsets of $Y_i$, and $d_i \in \BZ$.

Recall the finite group $\Gamma = \mathrm{Pic}^0(C)[n]$. For a $\Gamma$-action on a $\BC$-vector space $V$, we have the canonical decomposition
\[
V = V^{\Gamma} \oplus V_{\mathrm{var}}
\]
with $V^\Gamma$ the $\Gamma$-invariant part and $ V_{\mathrm{var}}$ the variant part. The following proposition concerns the $\Gamma$-actions on the intersection cohomology groups of $M_{n,L}$ and $N_{n,L}$.

\begin{prop} \label{prop4.1}
We have 
\[
\mathrm{dim} \mathrm{IH}^*(N_{n,L}, \BC)_\mathrm{var}  \leq 
\mathrm{dim} \mathrm{IH}^*(M_{n,L}, \BC)_\mathrm{var}.
\]
\end{prop}

\begin{proof}
We first show that the righthand side of the decomposition (\ref{decomp}) contains \[
\mathrm{IC}_{N_{n,L}}[\mathrm{dim}M_{n,L} - \mathrm{dim}N_{n,L}]
\]
as a direct summand component. Consider the open subvariety $M_{n,L}^{s} \subset M_{n,L}$ formed by stable Higgs bundles. By definition we have $M_{n,L}^s \cap N_{n,L} = N_{n,L}^s$ where $N_{n,L}^s$ is the locus of stable vector bundles. Both $M_{n,L}^s$ and $N_{n,L}^s$ are nonsingular. The component of the attracting locus $(M^s)^+$ over $N_{n,L}^s$ is an open subvariety of $M_{n,L}$, so we have the splitting over the stable locus $M_{n,L}^s$:
\[
(f^+)^*(g^+)^!\BC_{M_{n,L}^s}  \simeq  \BC_{N_{n,L}^s} \oplus \cdots.
\]
In particular, this shows that there is a term in the righthand side of (\ref{decomp}) with 
\[
Y_0 = N_{n,L}, \quad  \CL_0 = \BC, \quad d_0 = \mathrm{dim}M_{n,L} - \mathrm{dim}N_{n,L}.
\]
Hence (\ref{decomp}) induce an isomorphism
\begin{equation}\label{eq5}
\mathrm{IH}^*(M_{n,L}, \BC) \simeq \mathrm{IH}^*(N_{n,L}, \BC) \oplus \left( \bigoplus_{j>0} H^{*-\mathrm{dim}M_{n,L}+d_j}(F,~~\mathrm{IC}_{Y_j}(\CL_j))  \right).
\end{equation}
Since the $\Gamma$- and the $\BC^*$-actions on $M_{n,L}$ commute, the embeddings (\ref{fg}) are $\Gamma$-equivariant. The hyperbolic localization functor (\ref{functor}) and the isomorphisms (\ref{localization}) and (\ref{decomp}) are also $\Gamma$-equivariant. Consequently, (\ref{eq5}) is an $\Gamma$-equivariant isomorphism whose variant parts implies the proposition.
\end{proof}

\subsection{Codimension estimate}
Recall that $d_\gamma$ is the codimension of $A_\gamma$ in $A$. We have
\[
d_\gamma = \mathrm{dim}A - \mathrm{dim}A_\gamma = \mathrm{dim}A - \mathrm{dim}A(\pi)
\]
where $\pi: C' \to C$ is the \'etale Galois cover associated with $\gamma$. By the formulas of \cite[Section 6.1]{dC_SL} for the Hitchin bases, we obtain the following codimension formula for endoscopic loci.

\begin{lem}\label{lem4.2}
Assume that $\gamma \in \Gamma$ has order $m$ with $n = mr$. We have
\[
d_\gamma = \frac{n(n - r) \cdot \mathrm{deg}(D)}{2}.
\]
In particular for fixed rank $n$, we have $\mathrm{min}_{\gamma\neq 0}\{d_\gamma\}  \to +\infty$ when $\mathrm{deg}(D) \to \infty$.
\end{lem}

Now we complete the proof of Theorem \ref{thm0.3}.

\subsection{Proof of Theorem \ref{thm0.3}}

For fixed genus $g$ curve $C$ and rank $n$, we work with Higgs bundles with $\mathrm{deg}(D)$ even and large enough, so that $d_\gamma > \mathrm{dim}N_{n,L}$ for any nonzero $\gamma \in \Gamma$. This is possible due to Lemma \ref{lem4.2} and the fact that $\mathrm{dim}N_{n,L} = (n^2-1)(g-1)$ is independent of $\mathrm{deg}(D)$.


Theorem \ref{thm0.2} (a) implies that the variant part 
\begin{equation*}
\left(Rh_* \mathrm{IC}_{M_{n,L}}\right)_{\mathrm{var}} \in D^b_c(A)
\end{equation*}
(contributed by the nontrivial characters) is concentrated in degrees $\geq \mathrm{min}_{\gamma\neq 0}\{2d_\gamma\}$.\footnote{We note that this statement does not need the even assumption of $\mathrm{deg}(D)$.} Taking global cohomology, we have
\[
\mathrm{IH}^k(M_{n,L}, \BC)_{\mathrm{var}} = 0, \quad \quad \forall k < \mathrm{min}_{\gamma\neq 0}\{2d_\gamma\},
\]
which further yields from Proposition \ref{prop4.1} that
\[
\dim \mathrm{IH}^k(N_{n,L}, \BC)_{\mathrm{var}} \leq  \dim \mathrm{IH}^k(M_{n,L}, \BC)_\mathrm{var} =0, \quad \forall k < \mathrm{min}_{\gamma\neq 0}\{2d_\gamma\}.
\]
By our choice of $D$ we conclude that $\mathrm{IH}^*(N_{n,L}, \BC)_{\mathrm{var}} = 0$. This proves the triviality of the $\Gamma$-action on $\mathrm{IH}^k(N_{n,L}, \BC)$.

To prove (\ref{SL_PGL}), we consider the natural finite quotient map 
\[
f: N_{n,L} \to N_{n,L}/\Gamma= \check{N}_{n,L}.
\]
Since the intersection cohomology complex $\mathrm{IC}_{N_{n,L}}$ is naturally $\Gamma$-equivariant, the pushforward complex $f_*\mathrm{IC}_{N_{n,L}}$ admits a canonical decomposition with respect to the $\Gamma$-action:
\[
f_*\mathrm{IC}_{N_{n,L}} = \left(f_*\mathrm{IC}_{N_{n,L}}\right)^\Gamma \oplus \left(f_*\mathrm{IC}_{N_{n,L}}\right)_\mathrm{var}.
\]
By the first part of the theorem, the cohomology of $\left(f_*\mathrm{IC}_{N_{n,L}}\right)_\mathrm{var}$ vanishes. Therefore it suffices to show that the complex $\left(f_*\mathrm{IC}_{N_{n,L}}\right)^\Gamma$ coincides with $\mathrm{IC}_{\check{N}_{n,L}}$, which follows from Lemma \ref{lem3.5}. \qed

\end{document}